\documentclass[a4paper,12pt]{article}

\usepackage[top=1in,left=1in,right=1in,bottom=1.5in]{geometry}

\usepackage[T1]{fontenc}
\usepackage{amsmath,amssymb,amsfonts,amsthm,mathrsfs,eucal,dsfont,verbatim}

\usepackage[english]{babel}

\usepackage{color}

\newcommand{\rd}{{\mathbb{R}^d}}
\newcommand{\rone}{\mathbb{R}}
\renewcommand{\Re}{\rone}
\renewcommand{\star}{\circledast}

\newcommand{\E}{\mathds{E}}
\renewcommand{\P}{\mathds{P}}

\newcommand{\prt}{\partial}
\renewcommand {\epsilon}{\varepsilon}

\newcommand{\DD}{\mathbb{D}}

\newcommand{\eps}{\varepsilon}

\theoremstyle{plain}
\newtheorem{thm}{Theorem}[section]
\newtheorem{prop}{Proposition}[section]
\newtheorem{cor}{Corollary}[section]
\newtheorem{lem}{Lemma}[section]
\theoremstyle{definition}

\theoremstyle{remark}
\newtheorem{rem}{Remark}[section]

\usepackage{hyperref}

\DeclareMathSymbol{\ophi}{\mathalpha}{letters}{"1E}

\renewcommand{\phi}{\varphi}

\newcommand{\be}{\begin{equation}}
\newcommand{\ee}{\end{equation}}
\newcommand{\ben}{\begin{equation*}}
\newcommand{\een}{\end{equation*}}

\newcommand{\ba}{\begin{aligned}}
\newcommand{\ea}{\end{aligned}}

\newfont{\cyrfnt}{wncyr10}
\def\J3{\cyrfnt{\rm \u{\cyrfnt I}}}
\def\j3{\cyrfnt{\rm \u{\cyrfnt i}}}

\usepackage[]{color}
\definecolor{DarkGreen}{rgb}{0.1,0.7,0.3}   

\definecolor{DarkGreen}{rgb}{0.1,0.7,0.3}   

\numberwithin{equation}{section}

\allowdisplaybreaks[4]

\begin{document}

\title{On weak uniqueness and distributional properties of a solution to an SDE with $\alpha$-stable noise}

\author{%
        \textsc{Alexei Kulik}%
    \thanks{Institute of Mathematics, NAS of Ukraine, 3, Tereshchenkivska str., 01601  Kiev, Ukraine,
    \texttt{kulik.alex.m@gmail.com}}
    }

\date{}

\maketitle

\begin{abstract}
    \noindent
For an SDE driven by a rotationally invariant $\alpha$-stable noise we prove weak uniqueness of the solution under the balance condition $\alpha+\gamma>1$, where $\gamma$ denotes the H\"older index of the drift coefficient. We prove existence and continuity of the transition probability density  of the corresponding Markov process and give a representation of this density with  an explicitly given ``principal part'', and a ``residual part''  which possesses an upper bound. Similar representation is also provided for
the derivative of the transition probability density w.r.t. the time variable.


\end{abstract}

\section{Introduction}\label{s12}
In this paper we study the SDE
\be\label{SDE}
dX_t=b(X_t)\, dt+\sigma(X_{t-})\, dZ^{(\alpha)}_t,
\ee
 driven by a symmetric $\alpha$-stable process $Z^{(\alpha)}$ in $\Re^d$.  If $\alpha=2$, i.e. $Z^{(\alpha)}$ is a Brownian motion, it is well known that for \eqref{SDE} to have unique weak solution it is sufficient that $b$ is measurable and locally bounded and $\sigma$ is continuous and non-degenerate; see \cite{Stroock_Varad}, Chapter 7. In particular, it is a kind of a ``common knowledge'' that an SDE with H\"older continuous coefficients and non-degenerate diffusion coefficient possesses unique weak solution which, in addition, defines a time-homogeneous Markov process. This heuristically means that the Brownian noise possesses a kind of regularization feature: though the deterministic dynamics which corresponds to the drift part of \eqref{SDE} may fail to be  well defined, adding a non-degenerate diffusion part makes the entire stochastic dynamics well determined. One can expect that the same ``stochastic regularization'' effect should appear also in systems with more general L\'evy noises. However, even in a particular (but important) case of an SDE driven by an $\alpha$-stable noise substantially new phenomena may appear.

 Namely, if $\alpha<1$ and the H\"older index $\gamma$ of the drift coefficient $b(x)$ is positive but small, equation \eqref{SDE} may even fail to possess the weak uniqueness property. A natural example for this dates back to the paper \cite{TTW74}, where in the second part of Theorem 3.2 it is shown that for a one-dimensional SDE \eqref{SDE} with $\sigma(x)\equiv 1$, $b(x)=|x|^\gamma\mathrm{sign}\, x$,
 the minimal and the maximal (weak) solutions to \eqref{SDE} are different as soon as
 \be\label{non-balance}
\alpha+\gamma<1.
\ee
This example well illustrates the fact that  an SDE driven by an $\alpha$-stable process requires some specific tools, when compared with the diffusive one, to provide its solvability and to analyze the properties of the solution. The first step in this direction was made in   \cite{TTW74}, where the first part of Theorem 3.2 states the weak uniqueness  for a  one-dimensional SDE \eqref{SDE} with $\sigma\equiv 1$ and $\gamma$-H\"older continuous \emph{non-decreasing} drift coefficient under the additional condition
\be\label{balance}
\alpha+\gamma>1,
\ee
which in what follows we call the\emph{ balance condition}. The method of proof in \cite{TTW74} strongly relies both on the fact that $X$ is a gradient perturbation of $Z^{(\alpha)}$, and apparently can not be applied in a general setting.

In this paper we propose a method which makes it possible both to prove the weak uniqueness of the solution to \eqref{SDE} and to analyze  the properties of the transition probability density of the  corresponding Markov process. Our standing assumptions are that that the coefficient $\sigma$ is non-degenerate,  $b, \sigma$ are  H\"older continuous, and  the {balance condition} \eqref{balance} holds. We note that our  weak uniqueness result is quite sharp, since the only ``gap'' between \eqref{non-balance} and the balance condition \eqref{balance} is the ``critical case'' $\alpha+\gamma=1$. We postpone the study of the critical case for a further research; our conjecture is that the weak uniqueness in this case still holds true.  An important reference  concerning the transition probability density is the recent paper \cite{DF13}, where under the balance condition and the \emph{assumption}
of the weak uniqueness of the solution it was shown that the solution to SDE \eqref{SDE} possesses a distribution density and this density belongs to certain Besov space.  The class of L\'evy noises allowed in \cite{DF13} is wider than our $\alpha$-stable one; on the other hand, in this particularly important
case we give a much more detailed information about the transition probability density, especially about its small time  behavior. We also note that the weak uniqueness in the ``super-critical'' case $\alpha<1$ is a non-trivial  property which can not be derived from sufficient conditions available in the field so far; we postpone the detailed discussion  to Section \ref{over} below.

Let us explain the heuristics behind the balance condition. First, we note that for $\alpha\geq 1$ this   condition holds true for any H\"older continuous $b$, which makes this case similar to the diffusive one. This vaguely can be interpreted as follows: in this case, the stochastic part of the equation dominates the drift part. For $\alpha<1$ such a domination fails, and the balance condition \eqref{balance} reflects the necessity to cooperate somehow the partial regularity property of the drift with the regularization properties of the noise. Namely, if $b$ is $\gamma$-H\"older continuous, the ODE which corresponds to the deterministic part in \eqref{SDE} may fail to have the uniqueness properties, but it is still possible to bound the distance at time $t$ between any two solutions to the Cauchy problem with same initial conditions; see Section \ref{flows} below. This bound has the form $Ct^{1/(1-\gamma)}$, and the natural scale for the $\alpha$-stable process is $t^{1/\alpha}$. The balance condition  \eqref{balance}  is equivalent to $1/(1-\gamma)>1/\alpha$; that is, $Ct^{1/(1-\gamma)}\ll t^{1/\alpha}, t\to 0+$.
  Therefore  \eqref{balance} is essentially the condition for the noise to be ``intensive enough'' to make negligible an ``analytical uncertainty'' caused by the deterministic part of the equation.

Both our proof of the weak uniqueness and our estimates for the transition probability density are based on an analytical construction, which interprets the  transition probability density as a fundamental solution for the parabolic Cauchy problem associated with the (formal) generator of the Markov process $X$ defined by \eqref{SDE}, and provides this solution by means of a certain version of the \emph{parametrix method}. Some part  of this construction was developed recently in \cite{KK15}. In the modification of the parametrix method for the super-critical stable SDEs, proposed in  \cite{KK15}, Case \textbf{C}, the ``zero order approximation'' of the unknown fundamental solution combines the heat kernel for the stable part with the deterministic flow which corresponds to the drift term. The drift term  was supposed therein to be Lipschitz continuous, hence the respective deterministic flow was well defined. In the current paper we finalize this construction and develop the proper substitute for the  deterministic flow term, which is well defined for a H\"older continuous drift coefficient and still makes the entire parametrix construction operational. We also clarify the structure of the law of the solution to \eqref{SDE} in a small time. Namely, we show that the principal part the  probability density of $X_t$ with $X_0=x$ can be chosen as the law of
\be\label{trotter}
\widetilde X_{t,x}=\upsilon_t(x)+\sigma(x)Z_t^{(\alpha)},
\ee
where $\upsilon_t(x)$ denotes \emph{some} solution to the Cauchy problem for the ODE which corresponds to the deterministic part of the initial equation.

The paper is organized as follows. In Section~\ref{s20} we  formulate the main results of the paper and discuss the related results available in the literature.  Section~\ref{s3} is a preliminarily one for the proofs, and contains an outline of the parametrix method and constructions and estimates for approximate solutions to ODEs with H\"older continuous coefficients.  In Section~\ref{s4} the parametrix construction is specified, and the weak uniqueness of solution to \eqref{SDE} and  estimates for the transition probability density for this  solution  are proved. In Section~\ref{s5} we evaluate the properties of the derivative $\prt_tp_t(x,y)$.

\section{Main results}\label{s20}

In this section we formulate the main results of the paper.  The proofs  are postponed to the rest of the paper.

\subsection{Notation and assumptions}\label{sA}

 Through the paper we use the following notation. By $C_\infty(\rd)$ we denote the class of continuous functions vanishing at infinity; clearly, $C_\infty(\rd)$ is a Banach space with respect to the  $\sup$-norm $\|\cdot\|_\infty$. By $C_\infty^k(\rd)$, $k\geq 1$, (respectively, $C_b^k(\rd)$) we denote  the  class of $k$-times continuously differentiable functions, vanishing  at infinity  (respectively, bounded) together with their derivatives. In the case this does not cause misunderstandings we omit $\rd$ in the above notation; e.g. we often write $C_\infty$ instead of
 $C_\infty(\rd).$ As usual,   $a\wedge b :=\min(a,b)$, $a \vee b:=\max(a,b)$. By $|\cdot|$ we denote both the modulus of a real number and the Euclidean norm of a vector.  By  $c$  and $C$ we denote positive constants, the value of which may vary from place to place.   Relation $f\asymp g$ means that
$$
cg\leq f\leq C g.
$$
By $\Gamma(\cdot)$ we denote the Euler Gamma-function. We write $L_x f(x,y)$ in the case we need to  emphasize that an operator $L$ acts on a function $f(x,y)$ with respect to variable $x$. We use the following notation for space and time-space convolutions of functions respectively:
$$
(f\ast g)_t(x,y):=\int_{\Re^d}f_{t}(x,z)g_{t}(z,y)\, dz,
$$
$$
(f\star g)_t(x,y):=\int_0^t\int_{\Re^d}f_{t-s}(x,z)g_{s}(z,y)\, dzds.
$$

In what follows we specify the conditions on the objects involved in \eqref{SDE}, which we assume to hold true  throughout the entire paper.
Since our aim is to explain the main results and the methodology of the proofs in a most transparent way, we do not strive for imposing most general conditions possible. A the end of this subsection, a list of possible extensions is briefly discussed.

Let $Z^{(\alpha)}$, $\alpha\in (0,2)$,  be a L\'evy process  in $\rd$ with
$$
\E e^{i(\xi, Z_t^{(\alpha)})}=e^{-t|\xi|^\alpha}, \quad \xi\in \Re^d;
$$
that is, a rotationally invariant \emph{$\alpha$-stable} process.
 It is well known that the generator of the $C_\infty$-semigroup associated with $Z^{(\alpha)}$  is an extension of the operator
\begin{equation}\label{pv}
L^{(\alpha)}f(x)= \hbox{P.V.} \int_{\Re^d}\Big(f(x+u)-f(x)\Big)\frac{c_{\alpha,d}}{|u|^{d+\alpha}}du, \quad f\in C_\infty^2,
\end{equation}
where $c_{\alpha,d}$ is a constant which we do not need to specify here.  The  operator $L^{(\alpha)}$ is also called  a  \emph{fractional Laplacian}, and is denoted by $-(-\Delta)^{\alpha/2}$.

 By   $g^{(\alpha)}(x)$  we denote  the distribution density of the  variable $Z_1^{(\alpha)}$.
 Note that $L^{(\alpha)}$ is a homogeneous operator of the order $\alpha$ and the process $Z^{(\alpha)}$ is self-similar: for any $c>0$, the process
 $$
 c^{-1/\alpha}Z_{ct}^{(\alpha)},\quad t\geq  0
 $$
 has the same law as $Z^{(\alpha)}$. Consequently, the transition probability density of $Z^{(\alpha)}$ equals $t^{-d/\alpha} g^{(\alpha)}(t^{-1/\alpha} (y-x))$.

The drift coefficient $b:\rd\to \rd$ is assumed to be bounded and  H\"older continuous with the index $\gamma\in (0, 1]$: \begin{equation}\label{b-bdd-Hol}
|b(x)|\leq C, \quad |b(x)-b(y)|\leq C |x-y|^\gamma, \quad x,y\in \Re^d.
\end{equation}
The coefficient $\sigma$ is assumed to be scalar-valued. We denote by $a(x)=|\sigma(x)|^\alpha$ the \emph{jump intensity} coefficient. We assume this coefficient to be bounded and separated from zero and {H\"older continuous} with some index $\eta\in (0,1]$, i.e.
\begin{equation}\label{a_bdd_Hol}
 c\leq a(x)\leq C, \quad |a(x)-a(y)|\leq C|x-y|^\eta, \quad x,y\in \Re^d.
\end{equation}
Finally, consider the Cauchy problem to the ODE
\begin{equation}\label{forwardODE}
d\upsilon_t=b(\upsilon_t)\, dt, \quad \upsilon_0=x.
\end{equation}
By the Peano theorem, this problem has a solution, but if $\gamma<1$ such a solution may fail to be unique. Denote by
$\Upsilon(x)$ the set of all such solutions.

As we have already mentioned, most of the above assumptions can be weakened. For the weak uniqueness result, respective conditions on coefficients can hold true only locally. The coefficient $\sigma$ can be taken matrix-valued, and instead of the rotationally invariant stable noise one can consider a symmetric stable noise such that its spectral measure has a density w.r.t. the surface measure on unit sphere in $\rd$, which is bounded and bounded away from 0. A thorough proof of such an extension should require an extended version of Proposition \ref{A1} below; we do not discuss these technical issues here, in a separate research we plan to treat these issues in a maximal  generality.

\subsection{The main results}\label{main}

In all the results stated below we assume coefficients $b, a=|\sigma|^\alpha$ to satisfy \eqref{b-bdd-Hol}, \eqref{a_bdd_Hol} and the balance condition \eqref{balance} to hold true. Our first main result concerns the weak uniqueness of the solution to \eqref{SDE} and the basic properties of the corresponding Markov process.

\begin{thm}\label{t1} Equation \eqref{SDE} possesses unique weak solution $X_t, t\geq 0$, and this solution is a Markov process. This process is a Feller one; that is, it generates a strongly continuous semigroup $P_t, t\geq 0$ in $C_\infty$:
$$
P_tf(x)=\E_xf(X_t), \quad t\geq 0, \quad f\in C_\infty.
$$
The generator $(A,D(A))$ of this semigroup is an extension of the operator $(L, C^2_\infty)$ defined by
\begin{equation}\label{symbol}
Lf(x)=\Big(b(x), \nabla f(x)\Big)+a(x) L^{(\alpha)}f(x), \quad f\in C_\infty^2.
\end{equation}
\end{thm}

Our second main result concerns the properties of the transition probability density of the process $X$.

\begin{thm}\label{t2}
\begin{itemize}
  \item[I.] The Markov process $X$ has a transition probability density $p_t(x,y)$, i.e.
\be\label{P_int}
P_tf(x)=\int_{\rd}p_t(x,y)f(y)\, dy, \quad f\in  C_\infty, \quad t>0.
\ee
This density is a  continuous function  of $(t,x,y)\in (0, \infty)\times\rd\times\rd$.
  \item[II.] Fix a solution $\upsilon_\cdot(x)\in \Upsilon(x)$ of \eqref{forwardODE} and denote
 \begin{equation}\label{ptilde}
\widetilde p_t(x,y)= \frac{1}{t^{d/\alpha}a^{d/\alpha}(x)}g^{(\alpha)}\left({y-\upsilon_t(x)\over t^{1/\alpha}a^{1/\alpha}(x)}\right).
\end{equation}
 Denote by $\widetilde r_t(x,y)$ the residue term in the decomposition
\be\label{decomp}
 p_t(x,y)=\widetilde p_t(x,y)+\widetilde r_t(x,y),
\ee
 and put $\delta=1-1/\alpha+{\gamma/\alpha}$  (which is positive by the balance condition). Then
 for any  $\chi\in (0,\alpha\wedge \eta)$ and $T>0$, the following estimate for   the  residue term  holds true:
\begin{equation}\label{r_bound}
|\widetilde r_t(x,y)|\leq C\Big(t^\zeta+|y-\upsilon_t(x)|^\chi\wedge 1\Big)\widetilde p_t(x,y),\quad t\in (0,T], \quad x,y\in \rd,
\end{equation}
where
$$
\zeta=\min\left\{\delta, \chi, {\chi\over \alpha}\right\}.
$$
The transition probability density $p_t(x,y)$ itself possesses the following two-sided estimate:
 for any $T>0$, there exist positive $c,C$ such that
\be\label{two-sided}
c\widetilde p_t(x,y)\leq  p_t(x,y)\leq C\widetilde p_t(x,y), \quad t\in (0, T], \quad x,y\in \rd.
\ee
The constants  $c,C$ above does not depend on the choice of $\upsilon_\cdot(x)$; that is, the estimates \eqref{r_bound} and \eqref{two-sided} are  uniform over the class $\Upsilon(x)$.
\end{itemize}

\end{thm}
\begin{rem} It follows directly from \eqref{r_bound} and the formula for $\widetilde p_t(x,y)$ that
$$
\int_{\rd}|\widetilde r_t(x,y)|\, dy\leq Ct^\zeta.
$$
Hence the residue term in \eqref{decomp} is indeed negligible,  and $\widetilde p_t(x,y)$ represents the principal part of $p_t(x,y)$ as $t\to 0+$.
Note that the residue term is negligible in the integral sense but not uniformly, since the right hand side term in \eqref{r_bound}  is comparable to $\widetilde p_t(x,y)$ when $|y-\upsilon_t(x)|\geq 1$.
\end{rem}

\begin{rem}\label{r21} The decomposition \eqref{ptilde} will be obtained in two steps: first, by means of the parametrix method we will construct similar representation  with another principal part $p_t^0(x,y)$ (see \eqref{p01} below); second, we will make ``fine tuning'' of this representation in order to re-arrange the principal part.  The choice of the principal part $\widetilde p_t(x,\cdot)$ in \eqref{decomp} has a clearly seen advantage that it has a good  stochastic interpretation as  the distribution density of $\widetilde X_{t,x}$ defined by \eqref{trotter}.

On the other hand, the set of solutions $\Upsilon(x)$ is implicit and may have a complicated structure. Hence it might be useful for further applications  to have a representation of $p_t(x,y)$ in a form similar to \eqref{decomp}, but with $\upsilon_t(x)$ changed to a more explicit term. Here we briefly outline two possibilities to arrange such a representation. First, consider a sequence of Picard-type approximations
\be\label{Picard}
\upsilon_{0, t}(x)\equiv x, \quad \upsilon_{k, t}(x)=x+\int_0^tb(\upsilon_{k-1, s}(x))\, ds, \quad k\geq 1, \quad t\geq 0.
\ee
Though such a procedure now typically fails to give a successful approximation for a solution, it still can be used in a comprehensive representation of the transition probability density $p_t(x,y)$. Namely, denote
$$
\rho_k=1+\dots+\gamma^{k+1}-{1\over \alpha}, \quad k\geq 0,
$$
and assume that  for a given $k$
\be\label{211}
\rho_k>0.
\ee
We will show in Section \ref{s42} below that $p_t(x,y)$ have  a representation similar to \eqref{decomp} with $\widetilde p_t(x,y)$ replaced by
$$
\widehat p_t(x,y)= \frac{1}{t^{d/\alpha}a^{d/\alpha}(x)}g^{(\alpha)}\left({y-\upsilon_{k,t}(x)\over t^{1/\alpha}a^{1/\alpha}(x)}\right)
$$
and corresponding residue term $\widehat r_t(x,y)$ satisfying an analogue of \eqref{r_bound} with $\widetilde p_t(x,y)$ changed to $\widehat  p_t(x,y)$ and $\zeta$ changed to $\min\{\zeta, \rho_k/(1-\gamma)\}$.

Note that for any $k\geq 0$ condition \eqref{211} is strictly stronger than the balance condition \eqref{balance}, and under \eqref{balance} there exists $k$ such that \eqref{211} holds. We note that the heuristic explanation of the balance condition given in the Introduction well corresponds to condition \eqref{211}: if \eqref{211} holds, then the ``approximation error'' for $\upsilon_{k,t}(x)$ is $Ct^{1/\alpha+\rho_k}\ll t^{1/\alpha}, t\to 0+$ (see Section \ref{flows} below), hence the noise is ``intensive enough'' to negate this error. We mention that conditions \eqref{211} with  $k=0, k=1$ are exactly the assumptions imposed in \cite{KK15}, Case \textbf{A} and Case \textbf{B}, respectively.

Another possible option is to take an approximate solution $\upsilon_t^0(x)$ instead of $\upsilon_t(x)$, see Section \ref{flows} below; in this case $\zeta$ in the analogue of \eqref{r_bound} remains unchanged.
\end{rem}

In the last main result of this paper, we establish the properties of the derivative of the transition probability density w.r.t. the time variable. The aim of the latter theorem is two-fold. On one hand, the estimates on the derivative $\prt_tp_t(x,y)$ have natural applications in stochastic approximation problems, e.g. \cite{GL08}, \cite{KMN14}, \cite{GK14}. On the other hand, the proof of Theorem \ref{t3} below clarifies the methodology:  we will see below that the particular choice of the ``zero approximation term'' to $p_t(x,y)$ in the parametrix method is a subtle question, which can be solved in various ways. Several such ways may be successful if we are only interested in assertions of Theorem \ref{t1} and Theorem \ref{t2}; see Section \ref{choice}  below. Considering in addition the derivative $\prt_t$ actually shows that most of these choices are not suitable for the purposes of studying \emph{sensitivities} of this density. Here we consider the sensitivity w.r.t. $t$, only, but we expect that similar methods can be applied to other types of sensitivities  (e.g. w.r.t. $x,y$ or w.r.t. additional parameters involved into the coefficients). This is the subject of our forthcoming research.

\begin{thm}\label{t3}  The function $p_t(x,y)$ possesses a continuous derivative $\prt_tp_t(x,y)$ on the set $(0, \infty\times \rd\times \rd)$. Denote by $\widetilde R_t(x,y)$ the residue term in the representation
\be\label{decomp_dt}
 \prt_t p_t(x,y)=\prt_t\widetilde p_t(x,y)+ \widetilde R_t(x,y),
\ee
where $\widetilde p_t(x,y)$ is given by \eqref{ptilde}, and denote
$$\alpha'=\min\{1, \alpha\}.
$$
Then for any $T>0$ and $\chi\in (0,\alpha\wedge \eta)$ the following estimates for the principal  and the residue terms in \eqref{decomp_dt} hold true:
$$
|\prt_t\widetilde p_t(x,y)|\leq Ct^{-1/\alpha'} \widetilde p_t(x,y),\quad t\in (0, T], \quad x,y\in \rd,
$$
$$
 |\widetilde R_t(x,y)|\leq Ct^{-1/\alpha'}\Big(t^\zeta+|y-\upsilon_t(x)|^\chi\wedge 1\Big)\widetilde p_t(x,y),\quad t\in (0, T], \quad x,y\in \rd.
$$

As a corollary, the derivative  $\prt_tp_t(x,y)$ itself possesses the estimate
\be\label{der_bound}
|\prt_t p_t(x,y)|\leq Ct^{-1/\alpha'} \widetilde p_t(x,y),\quad t\in (0, T], \quad x,y\in \rd,
\ee

\end{thm}
We mention that if the original equation does not contain the drift ($b(x)\equiv 0$) then in the above estimates $\alpha'$ can be changed to 1.
We also remark that the statement of Theorem \ref{t3} is even stronger than the one formulated in Theorem 2.6 \cite{KK15}, Case \textbf{C} under the stronger assumption that $b$ is Lipscghitz continuous. Such an improvement is a result of a well chosen ``zero approximation term''   in the parametrix construction.

\subsection{Related results: an overview}\label{over}

   The weak uniqueness problem for L\'evy driven SDEs is closely related to the well-posedness of the martingale problem for  integro-differential operators of certain type. A large  group of papers in that direction is  available, e.g.  \cite{Ba88},\cite{Ko84a}, \cite{Ko84b},  \cite{MP92a}--\cite{MP12}, this list is far form being complete. The weak uniqueness results available to the author, in various forms, rely on a typical assumption which vaguely means that the ``jump part'' of the operator dominates the entire operator (which, from the analytical point of view, is a kind of the sectorial condition).   For the formal generator \eqref{symbol} with $\alpha<1$ and non-zero drift, this condition fails and the SDE \eqref{SDE} in that case is far away from the domain where the available weak uniqueness results are applicable. Of course, if $\alpha\in (1,2)$ the weak uniqueness follows easily from the available results, e.g.  \cite{Ba88}.

An interesting counterpart to our weak uniqueness result is contained in the recent preprint \cite{CSZ15}, where the  \emph{strong} uniqueness for an SDE of the form \eqref{SDE} with H\"older continuous drift coefficient is proved under the following assumption which looks similar to the balance condition \eqref{balance}: $\alpha+\gamma/2>1$. Though, the results of \cite{CSZ15} are not fully comparable with ours because therein  $\sigma(x)\equiv 1$.

Our proof of the weak uniqueness is based on the parametrix construction of an (approximate) fundamental solution to a Cauchy problem for $\prt_t-L$. The argument  is insensitive w.r.t. the structure of the model, and is actually  based on the fact that, as soon as the parametrix construction is completed, one can construct a large family of \emph{approximate harmonic functions} for the operators $\prt_t-L,\prt_t+L$; see Section \ref{s43}.   We feel that the idea behind this argument is close to the one developed in the diffusion setting in \cite{BP09} and extended in \cite{M11}, \cite{HM14}, though the particular form of the argument is different. We mention also that using properly chosen approximate harmonic functions one can extend the classical argument based on the Positive Maximum Principle for $L$ in order to get the positivity and other properties of the semigroup $P_t, t\geq 0$; see Section 4 in \cite{KK15}. In the current setting this ``analytical'' argument applies as well, but in order to explain all available possibilities we give another proof. This ``probabilistic''  proof is shorter but less explicit and  requires more preliminaries about  Markov processes being  solutions to a martingale problem.

For the background of the parametrix construction in the classical diffusive setting, we refer to the monograph by Friedman \cite{Fr64}; see also the original paper by E.Levi \cite{Le07} and the paper by W. Feller \cite{Fe36}. This construction was  extended to equations with pseudo-differential operators in   \cite{ED81}, \cite{Ko89} and \cite{Ko00}, see also the reference list and an  extensive overview in the monograph \cite{EIK04}.  The list of subsequent and related  publications is large, and we cannot discuss it here  in details. Let us only mention three recent papers: \cite{KM11}, where the discrete-time analogue of the parametrix construction for the Eurer scheme for stable-driven SDEs was developed,  \cite{CZ13},  where two-sided estimates, more precise than those in \cite{Ko89}, were obtained,   and \cite{BK14},  where the  probabilistic interpretation of the parametrix construction and its application to the Monte-Carlo  simulation was developed.

    In all the references listed above it is required that either the stability index $\alpha$ is $>1$, or the gradient term is not  involved in the equation: this is the same ``sectorial type'' assumption which was mentioned above. If $\alpha<1$, because of the lack of domination of the ``jump part'' of the generator, the proper construction of the ``zero approximation term'' should involve an additional correction which corresponds to the drift, two versions of such a construction were proposed in \cite{KK15}, Case \textbf{B} and Case \textbf{C}. This effect has a similar nature with the one revealed in \cite{M11}, where a chain of equations is considered, where only the last equation contains the diffusive term and which consequently corresponds to  a degenerate diffusion. Because of the degeneracy, the diffusive term therein also lacks the domination property, and this motivates extra correction terms  in the parametrix construction. We mention also the recent preprints \cite{HM14}, where a chain of equations driven by a stable process is studied in a similar manner, and \cite{H15} where an SDE driven by tempered $\alpha$-stable process with $\alpha<1$ and possibly singular spectral measure is considered. In all these references the drift coefficient is assumed to be Lipschitz continuous, which makes corresponding ODEs for correcting terms to be well solvable. The only exception is the Case B in \cite{KK15}, however  therein a condition stronger  than the balance condition  \eqref{balance} is imposed on the H\"older index for $b$.

For brevity, we omit a discussion here and refer to \cite{KK15} for other related topics: the heuristics for the choice of the zero-order term in the parametrix expansion, the related papers  which concern SDEs with singular drift terms (\cite{Po94}, \cite{PP95}, \cite{BJ07}, \cite{KS14}), the related results of \cite{KS14} and \cite{CW13} on weak uniqueness for gradient perturbations of stable generators, and the large group of  results, focused on the construction of a \emph{semigroup} for a Markov process with a given symbol rather than of the transition probability density $p_t(x,y)$, which relies on the symbolic calculus approach for the  parametrix  construction  (\cite{Ja94}, \cite{Ja96},  \cite{Ho98a}, \cite{Ho98b},  \cite{Bo05}, \cite{Bo08},   \cite{Ku81},  \cite{Ts74}, \cite{Iw77}).

\section{Preliminaries to the proofs}\label{s3}

In this section we outline the parametrix method, which is our key tool for proving the main statements given in Section \ref{main}. We also develop an auxiliary construction and  evaluate some results about approximate solutions to ODEs with H\"older continuous coefficients, which will be used in the subsequent proofs.

\subsection{The parametrix construction: an outline}

If $X$ is a Markov process solution to \eqref{SDE}, by the It\^o formula one can naturally expect that its generator $(A, D(A))$ is an extension of $(L, C^2_\infty)$; the operator $L$ is defined in \eqref{symbol}. Then one can try to seek for the unknown transition probability density of $X$ assuming it is a  \emph{fundamental solution} to the parabolic Cauchy problem associated with $L$. Recall that a function $p_t(x,y)$ is said to be a {fundamental solution} to the Cauchy problem for an operator
\be\label{L_ful}
\prt_t-L,
\ee
if for  $t>0$ it is differentiable in $t$, belongs to the domain of $L$ as a function of $x$, and satisfies
\be\label{L_fund}
\Big(\prt_t-L_x\Big)p_t(x,y)=0, \quad t>0, \quad x,y\in \Re^d,
\ee
\be\label{L_delta}p_t(x, \cdot)\to \delta_x, \quad t\to 0+, \quad x\in \Re^d;
\ee
see  \cite[Def.~2.7.12]{Ja02} in the case of a general pseudo-differential operator, which is the generalization of the corresponding definition (cf. \cite{Fr64}, for example) in the parabolic/elliptic setting.

A classical tool for constructing fundamental solutions is the  \emph{parametrix method}, below we explain the version of this method which is used in the sequel. Fix some function $p_t^0(x,y)$, which will be considered as a ``zero order approximation'' to the unknown fundamental solution $p_t(x,y)$. Denote by $r_t(x,y)$ the residue term  with respect to  this approximation:
\be\label{sol}
p_t(x,y)=p_t^0(x,y)+r_t(x,y).
\ee
If $p_t^0(x,y)$ belongs to $C^1$ and $C^2_\infty$ in the variables $t$ and $x$, respectively, we can put
\be\label{Phi}
\Phi_t(x,y):=-\Big(\prt_t-L_x\Big)p_t^0(x,y),\quad t>0, \quad x,y\in\Re^d.
\ee
Because $p_t(x,y)$ is supposed  to be the fundamental solution for the operator  (\ref{L_ful}) and $A$ extends $L$, we have
$$
\Big(\prt_t-L_x\Big)r_t(x,y)=\Phi_t(x,y).
$$
If $p_t^0(x,y)$ satisfies an analogue of \eqref{L_delta}, then a formal solution to this equation can be given in  terms of  the unknown fundamental solution $p_t(x,y)$,  and then using (\ref{sol}) we get the following equation for $r_t(x,y)$:
$$
r_t(x,y)=(p\star \Phi)_t(x,y)=(p^0\star \Phi)_t(x,y)+(r\star \Phi)_t(x,y).
$$
The formal solution to this equation is given by the convolution
\be\label{r}
r_t(x,y)=(p^0\star \Psi)_t(x,y),
\ee
where $\Psi$ is the sum of $\star$-convolution powers  of $\Phi$:
\be\label{Psi}
\Psi_t(x,y)=\sum_{k\geq 1}\Phi^{\star k}_t(x,y).
\ee
If  the series (\ref{Psi})  converges and the convolution  (\ref{r}) is well defined, we obtain   the required function $p_t(x,y)$ in the form
\be\label{sol_1}
p_t(x,y)=p_t^0(x,y)+\sum_{k\geq 1}(p^0\star\Phi^{\star k})_t(x,y).
\ee

Clearly, the above argument is yet purely formal; to make it rigorous, we need to prove that  the parametrix construction is feasible, i.e. that the sum in the r.h.s. of (\ref{sol_1}) is well defined, and then to associate  $p_t(x,y)$  with the initial operator $L$. This last step is far from being trivial, hence we postpone its discussion to Section \ref{s43}. Here we just mention that  $p_t(x,y)$ which we actually obtain is not a fundamental solution in the classical sense exposed above. However, it still can be interpreted as an (approximate) fundamental solution in a sense, which is completely sufficient for all our  purposes.
The first step is more direct, and we give here a generic calculation which we use to analyse the convolution powers involved in \eqref{Psi} in a unified  way.

We say that a non-negative kernel $\{H_{t}(x,y), t>0, x,y\in \Re^d\}$ has a \emph{sub-convolution property}, if for every $T>0$ there exists a constant $C_{H,T}>0$ such that
\begin{equation}\label{H0}
(H_{t-s}* H_s)(x,y)\leq C_{H,T} H_{t}(x,y), \quad t\in (0, T], \quad s\in (0, t), \quad x,y\in \Re^d.
\end{equation}

For the following general estimate we refer to \cite{KK15}, Lemma 3.2 (the estimate \eqref{F25} below slightly differs from (3.29) in \cite{KK15}, but its proof is completely the same).
\begin{lem}\label{lH10}
Suppose that  function $\Phi_t(x,y)$  satisfies
\begin{equation}\label{F10}
\big|\Phi_t(x,y)\big| \leq C_{\Phi,T} \Big( t^{-1+\delta_1} H_t^1(x,y)+ t^{-1+\delta_2} H_t^2(x,y)\Big), \quad t\in (0,T], \, x,y\in \rd,
\end{equation}
with some $\delta_1, \,\delta_2\in (0,1)$ and some non-negative kernels $H_t^1(x,y), H_t^2(x,y).$ Assume also that
the kernels $H_t^i(x,y)$, $i=1,2$, have  the sub-convolution property with constant $C_{H,T}$,  and
\begin{equation}\label{H12}
H_t^1(x,y)\geq H_t^2(x,y).
\end{equation}

Then  for any $t\in (0,T]$, $x,y\in \rd$, we have

\begin{itemize}
\item[a)]
\begin{equation}\label{Fk}
\Big| \Phi^{\star k}_t(x,y)\Big| \leq  \frac{C_1C_2^k}{\Gamma(k\zeta)}
t^{-1+(k-1)\zeta} \Big(t^{\delta_1}H_t^{1}(x,y) + t^{\delta_2}  H_t^2(x,y)\Big),\quad k\geq 2,
\end{equation}
where
\begin{equation}\label{const}
C_1 = (3C_{H,T})^{-1}, \quad C_2= 3C_{\Phi,T} C_{H,T} \Gamma(\zeta),\quad \text{ and}\quad
\zeta=\delta_1\wedge \delta_2;
\end{equation}

\item[b)] the series $\sum_{k=1}^\infty \Phi_t^{\star k}(x,y)$ is absolutely convergent and
\begin{equation}\label{F20}
\Big|\sum_{k=1}^\infty \Phi_t^{\star k}(x,y)\Big|\leq C  \Big(t^{-1+\delta_1} H_t^1(x,y)+ t^{-1+\delta_2} H_t^2(x,y)\Big);
\end{equation}

\item[c)]
\begin{equation}\label{F25}
\Big| \Big( H^2 \star \sum_{k=1}^\infty \Phi_t^{\star k}\Big)(x,y)\Big|\leq C   \Big(t^{\delta_1}H_t^{1}(x,y) + t^{\delta_2}  H_t^2(x,y)\Big).
\end{equation}
\end{itemize}
\end{lem}

Hence to make  the parametrix construction  feasible it is sufficient to choose  zero order approximation $p^0_t(x,y)$ in such a way that
the corresponding function $\Phi$, defined by \eqref{Phi}, satisfies \eqref{F10} and
\be\label{F11}
p^0_t(x,y)\leq C_{T}H^2_t(x,y), \quad t\in(0, T], \quad x,y\in \rd.
\ee
If the kernels $H^1,H^2$ have the sub-convolution property and satisfy \eqref{H12}, it will follow from Lemma \ref{lH10} that all the convolution powers in \eqref{Psi} are well defined, the series converges, and the residue term (i.e. the sum of the series in the right hand side of \eqref{sol_1}) possesses an upper bound of the form
\be\label{boundr}
|r_t(x,y)|\leq C  \Big(t^{\delta_1}H_t^{1}(x,y) + t^{\delta_2}  H_t^2(x,y)\Big).
\ee

\subsection{Approximate solutions to ODEs with H\"older continuous coefficients}\label{flows}

The proper choice of the zero order approximation $p^0_t(x,y)$ in the construction outlined above is a delicate point. In \cite{KK15}, in the case of a Lipschitz continuous drift coefficient $b$, this approximation was chosen in the form
\begin{equation}\label{p0}
p_t^0(x,y)= \frac{1}{t^{d/\alpha}a^{d/\alpha}(y)}g^{(\alpha)}\left({\theta_t(y)-x\over t^{1/\alpha}a^{1/\alpha}(y)}\right),
\end{equation}
with $\theta_t(y)$ being the unique solution to the Cauchy problem
\begin{equation}\label{backwardODE}
d\theta_t=-b(\theta_t)\, dt, \quad \theta_0=y.
\end{equation}
We will use similar $p_t^0(x,y)$ in the current setting, where $b$ is assumed to be H\"older continuous, only. Because now  \eqref{backwardODE} may have multiple solutions, this leads to a necessity to choose the ``correcting term'' $\theta_t(y)$ more carefully. In the current section we provide auxiliary construction and give some results which will be used in such a choice and the subsequent analysis  of the series \eqref{sol_1}.

Consider the following ``mollified'' family generated by the drift coefficient $b$:
$$
b(t,x)=(2\pi)^{-d/2} t^{-d/\alpha}\int_{\rd}b(z)e^{-|z-x|^2/2t^{2/\alpha}}\, dz,\quad t\geq 0.
$$
Because $b$ is $\gamma$-H\"older continuous, we have
\be\label{appr}
|b(x)-b(t,x)|\leq C t^{-d/\alpha}\int_{\rd}|z-x|^\gamma e^{-|z-x|^2/2t^{2/\alpha}}\, dz\leq C t^{\gamma/\alpha}.
\ee
On the other hand, for positive $t$ respective $b(t, \cdot)$ is smooth and
$$\ba
\nabla b(t,x)&=(2\pi)^{-d/2} t^{-d/\alpha}\int_{\rd}b(z)\left(z-x\over t^{2/\alpha}\right)e^{-|z-x|^2/2t^{2/\alpha}}\, dz
\\&=(2\pi)^{-d/2} t^{-d/\alpha}\int_{\rd}\big(b(z)-b(x)\big)\left(z-x\over t^{2/\alpha}\right)e^{-|z-x|^2/2t^{2/\alpha}}\, dz.
\ea
$$
Hence
\be\label{appr_der}
|\nabla b(t,x)|\leq C t^{-d/\alpha-2/\alpha}\int_{\rd}|z-x|^{1+\gamma}e^{-|z-x|^2/2t^{2/\alpha}}\, dz\leq C t^{\gamma/\alpha-1/\alpha}.
\ee
In particular, $b_t$ satisfies the Lipschitz condition with the Lipschitz constant $L(t)=Ct^{\gamma/\alpha-1/\alpha}$. Observe that by the balance condition \eqref{balance},
\be\label{Lip_int}
\int_0^T L(t)\, dt<\infty, \quad T>0.
\ee
Then for any fixed $s\in \Re$ the following Cauchy problems have unique solutions:
$$
{d\over dt}\upsilon_t=b(|t-s|, \upsilon_t), \quad \upsilon_0=x,
$$
$$
{d\over dt}\theta_t=-b(|t-s|, \theta_t), \quad \theta_0=y.
$$
We denote these solutions by $\upsilon^s_t(x)$ and $\theta^s_t(y)$, respectively.

For any $\upsilon_\cdot(x)\in \Upsilon(x)$ and $0\leq s\leq t\leq T$ we have
$$\ba
\upsilon_t(x)-\upsilon_t^s(x)&=\int_0^t\Big(b(\upsilon_r(x))-b(|r-s|, \upsilon_r^s(x))\Big)\, dr\\&=\int_0^t\Big(b(|r-s|, \upsilon_r(x))-b(|r-s|, \upsilon_r^s(x))\Big)\, dr+h^s_t(x),
\ea
$$
where by \eqref{appr}
$$
|h_t^s(x)|=\left|\int_0^t\Big(b(\upsilon_r(x))-b(|r-s|, \upsilon_r(x))\Big)\, dr\right|\leq Ct^{1+\gamma/\alpha}.
$$
Then by \eqref{Lip_int} and the Gronwall inequality,
\be\label{appr_sol}
|\upsilon_t(x)-\upsilon_t^s(x)|\leq Ct^{1+\gamma/\alpha}=Ct^{1/\alpha+\delta},
\ee
see the statement II of Theorem \ref{t2} where $\delta$ is defined. This means  that $\upsilon_t^s(x)$ approximates the value $\upsilon_t(x)$ for any solution $\upsilon_\cdot(x)\in \Upsilon(x)$  with the accuracy $Ct^{1/\alpha+\delta}$  Clearly, it follows from \eqref{appr_sol} that for any $t\in [0, T], s,r\in [0,t]$
\be\label{appr_sol_1}
|\upsilon_t^s(x)-\upsilon_t^r(x)|\leq Ct^{1/\alpha+\delta}.
\ee
Similarly, we have  for any $t\in [0, T], s,r\in [0,t]$
\be\label{appr_sol_2}
|\theta_t^s(x)-\theta_t^r(x)|\leq Ct^{1/\alpha+\delta}.
\ee
To explain the heuristics behind the following lemma, assume for a while that $b$ is Lipschitz continuous, then
the usual argument based on the uniqueness properties for corresponding ODEs shows that $\{\upsilon_t\}$ and $\{\theta_t\}$ are mutually inverse flows of solutions to \eqref{forwardODE} and \eqref{backwardODE}, respectively. Moreover, each $\upsilon_t,\theta_t:\rd\to \rd$ is a mapping which differs from the identity by function which satisfies Lipschitz conbdition with the constant $Ct$, hence for any $t\in [0, T], s\in [0, t]$
\be\label{flooow}
e^{-Ct}|x-\theta_t(y)|\leq |\upsilon_{t-s}(x)-\theta_s(y)|\leq e^{Ct}|x-\theta_t(y)|.
\ee
Exactly this inequality was used in \cite{KK15} as a key ingredient in the proof of the sub-convolution property of the kernels associated with $p_t^0(x,y)$ defined by \eqref{p0}. In the current setting, we will use the following analogue of this  inequality for the family of approximate solutions to \eqref{forwardODE} and \eqref{backwardODE}.

\begin{lem} For any $T>0$, there exists $C>0$ such that for any $t\in [0, T], s\in [0,t]$
\be\label{aflow}
e^{-Ct^\delta}|x-\theta_t^0(y)|-Ct^{1/\alpha+\delta}\leq |\upsilon_{t-s}^{t-s}(x)-\theta_s^0(y)|\leq e^{Ct^\delta}|x-\theta_t^0(y)|+Ct^{1/\alpha+\delta}.
\ee
\end{lem}
\begin{proof}

It follows from \eqref{appr_der} and the fact that $L(t)$ is locally integrable, that $\upsilon_t^r(x)$ is differentiable in $x$, and its derivative satisfies the linear ODE
$$
{d\over dt}(\nabla \upsilon_t^r(x))=(\nabla b)(|t-r|, \upsilon_t^r(x))(\nabla \upsilon_t^r(x)), \quad \nabla \upsilon_0^r(x)=I_{\rd}.
$$
This implies that for $t\leq T$, $r\in [0, t]$
$$
|\nabla \upsilon_t^r(x)-I_{\rd}|\leq C\int_0^t|s-r|^{\gamma/\alpha-1/\alpha}\, ds\leq Ct^{1+\gamma/\alpha-1/\alpha}=Ct^\delta,
$$
and therefore for any $x, x'$
$$
e^{-Ct^\delta}|x-x'|\leq |\upsilon_t^r(x)-\upsilon_t^r(x')|\leq e^{Ct^\delta}|x-x'|.
$$
For $s,t$ being fixed, take in the above inequality $t=t-s, r=r-s$:
\be\label{compar}
e^{-Ct^\delta}|x-x'|\leq |\upsilon_{t-s}^{t-s}(x)-\upsilon_{t-s}^{t-s}(x')|\leq e^{Ct^\delta}|x-x'|.
\ee
Next, consider the approximate solution $f(\cdot):=\upsilon_\cdot^{t-s}(\theta_t^s(y))$,  it is easy to see that $g(\cdot)=f(t-\cdot)$ satisfies
$$
g'(\tau)=b(|\tau-s|, g(\tau)), \quad g(t)=\theta_t^s(y).
$$
The function $\theta_\cdot^s(y)$ satisfies the same  Cauchy problem, and because the solution to this problem is unique we have
$$
\upsilon_{t-s}^{t-s}\Big(\theta_t^s(y)\Big)=\theta_s^s(y).
$$
Taking in \eqref{compar} $x'=\theta_t^s(y)$ and recalling that by \eqref{appr_sol_2} we have
$$
|\theta_t^0(y)-\theta_t^s(y)|\leq Ct^{1/\alpha+\delta}, \quad |\theta_s^0(y)-\theta_s^s(y)|\leq Ct^{1/\alpha+\delta},
$$
we complete the proof of \eqref{aflow}.
\end{proof}

Finally, we briefly discuss the Picard-type iteration procedure \eqref{Picard} for the ODE \eqref{forwardODE}; see the notation  introduced in Remark \ref{r21}.

Define
$$
\upsilon_{0, t}(x)\equiv x, \quad \upsilon_{k, t}(x)=x+\int_0^tb(\upsilon_{k-1, s}(x))\, ds, \quad k\geq 1, \quad t\geq 0.
$$
For any $\upsilon_\cdot(x)\in \Upsilon(x)$, we have
$$
|\upsilon_{0, t}(x)-\upsilon_t(x)|\leq Ct
$$
and
$$
|\upsilon_{k, t}(x)-\upsilon_t(x)|\leq C\int_0^t|\upsilon_{k-1, s}(x)-\upsilon_s(x)|^\gamma\, ds, \quad k\geq 1.
$$ Then it is  easy to show by induction that
\be\label{329}
|\upsilon_{k, t}(x)-\upsilon_t(x)|\leq Ct^{1+\gamma+\dots+\gamma^k}=Ct^{1/\alpha+\rho_k}.
\ee

\section{Proofs of Theorem~\ref{t1} and Theorem~\ref{t2}} \label{s4}

\subsection{The parametrix construction: the details}\label{s41}

Our first step in the proof of Theorem~\ref{t1} and Theorem~\ref{t2} is to specify the parametrix construction outlined before; that is, to choose $p_t^0(x,y)$ and to prove that the sum of the series \eqref{sol_1} is well defined. We denote $\kappa_t=\theta_t^0$ and put
\begin{equation}\label{p01}
p_t^0(x,y)= \frac{1}{t^{d/\alpha}a^{d/\alpha}(y)}g^{(\alpha)}\left({\kappa_t(y)-x\over t^{1/\alpha}a^{1/\alpha}(y)}\right).
\end{equation}
That is, the structure of $p^0_t(x,y)$ is similar to the one proposed in \eqref{p0}, but an exact solution to \eqref{backwardODE} therein is replaced by an approximate solution $\theta_t^s$ with $s=0$. Such a choice of $p^0_t(x,y)$ is not the only possible, and it is far from being evident which choice is the best. We discuss this subtle point in Section \ref{choice}.

Below we evaluate the kernel $\Phi$  and give an upper bound for it. The calculations here are similar to those made in Section 3.1 \cite{KK15}. In order to make the exposition self-sufficient, we explain the key points of the main argument, referring for technicalities to \cite{KK15}.

First, we formulate auxiliary statements we use in the calculation. For the proofs we refer to Appendix A, \cite{KK15}. Denote for
 $\lambda >  0$
\begin{equation}\label{Gbet}
G^{(\lambda)}(x) = \big( |x|\vee 1)^{-d-\lambda}, \quad x\in\rd.
\end{equation}

\begin{prop}\label{A1}
\begin{enumerate}
  \item For any $\lambda>0$, $c>0$ there exists $C>0$ such that
\begin{equation}\label{G1}
G^{(\lambda)}(cx)\leq C G^{(\lambda)}(x).
\end{equation} \item For any  $\lambda_1>\lambda_2$,
\begin{equation}\label{G2}
 G^{(\lambda_1)}(x)\leq  G^{(\lambda_2)}(x).
\end{equation}
  \item For any $\eps\in (0, \lambda)$,
\begin{equation}\label{G3}
|x|^{\eps} G^{(\lambda)}(x)\leq C G^{(\lambda-\eps)}(x).
\end{equation}
\item For any $\alpha\in (0,2)$,
\begin{equation}\label{g_a}
g^{(\alpha)}(x)\asymp G^{(\alpha)}(x),
\end{equation}
\begin{equation}\label{g_a_der}
\Big|(\nabla g^{(\alpha)})(x)\Big|\leq C G^{(\alpha+1)}(x),
\end{equation}
\begin{equation}\label{g_a_frac}
\Big|(L^{(\alpha)} g^{(\alpha)})(x)\Big|\leq C G^{(\alpha)}(x),
\end{equation}
\begin{equation}\label{g_a_frac_der}
\Big|(\nabla L^{(\alpha)} g^{(\alpha)})(x)\Big|\leq C G^{(\alpha+1)}(x),
\end{equation}
\begin{equation}\label{g_a_der2}
\Big|(\nabla^2 g^{(\alpha)})(x)\Big|\leq C G^{(\alpha+2)}(x).
\end{equation}
\end{enumerate}
\end{prop}

Denote
\begin{equation}\label{Qt21}
Q_t^{(\lambda)}(x,y):=\left( \Big|\frac{\kappa_t(y)-x}{t^{1/\alpha}}\Big|^\lambda  \wedge t^{-\lambda/\alpha} \right)\frac{1}{t^{d/\alpha}} G^{(\alpha)} \left({\kappa_t(y)-x\over t^{1/\alpha}}\right), \quad \lambda\in [0,\alpha).
\end{equation}

\begin{lem}\label{lPhi1}
Let $\chi\in (0,\alpha\wedge \eta), T>0$.  Then
\begin{equation}\label{Phi1}
|\Phi_t(x,y)|\leq C\Big( t^{-1+\chi/\alpha}Q_t^{(\chi)} (x,y)+ (t^{-1+\chi} +t^{-1+\delta}) Q_t^{(0)}(x,y)\Big), \quad t\in (0,T], \quad x,y\in \Re^d.
\end{equation}
\end{lem}

\begin{proof} To improve the readability, here and below we assume that $T>0$ is fixed and, if it is not stated otherwise, in any formula containing $t,x$, or $y$  we assume $t\in (0,T], x\in \Re^d, y\in \Re^d$.

Operators $\nabla$ and $L^{(\alpha)}$ are homogeneous with respective orders $1$ and  $\alpha$. From the identity
$$
(\partial_t - L^{(\alpha)})\big[t^{-d/\alpha} g^{(\alpha)}(t^{-1/\alpha}x)\Big] =0,
$$
we derive
\begin{align*}
\prt_tp_t^0(x,y)&=\Big[a(y) \frac{1}{a^{d/\alpha}(y)t^{d/\alpha} } (L^{(\alpha)} g^{(\alpha)})\left(\frac{w}{a(y)t^{1/\alpha}} \right)\\
&\quad +\Big(\prt_t\kappa_t(y), \frac{1}{a^{d/\alpha+1}(y)t^{d/\alpha+1} } (\nabla g^{(\alpha)})\left(\frac{w}{a(y)t^{1/\alpha}} \right)\Big)\Big]\Big|_{w=\kappa_t(y)-x}.
\end{align*}
On the other hand,
\begin{align*}
L_xp_t^0(x,y)&=\Big[a(x)\frac{1}{a^{d/\alpha}(y)t^{d/\alpha} } (L^{(\alpha)} g^{(\alpha)})\left(\frac{w}{a(y)t^{1/\alpha}} \right)\\
&\quad \quad -\Big(b(x), \frac{1}{a^{d/\alpha+1}(y)t^{d/\alpha+1} } (\nabla g^{(\alpha)})\left(\frac{w}{a(y)t^{1/\alpha}} \right)\Big)\Big]\Big|_{w=\kappa_t(y)-x}.
\end{align*}
Because $\prt_t\kappa_t(y)=-b(t, \kappa_t(y))$, we  finally  get
\begin{equation}\label{Phi_c}
\begin{split}
\Phi_t(x,y)&=\big(  L_x-\partial_t\big) p_t^0(x,y)\\
&=\Big(a(x)-a(y)\Big) {1\over t^{d/\alpha+1}a^{d/\alpha+1}(y)}(L^{(\alpha)} g^{(\alpha)})\left({\kappa_t(y)-x\over t^{1/\alpha}a^{1/\alpha}(y)}\right)\\&+
{1\over t^{(d+1)/\alpha}a^{(d+1)/\alpha}(y)}\left(b(t, \kappa_t(y))-b(x), (\nabla g^{(\alpha)})\left({\kappa_t(y)-x\over t^{1/\alpha}a^{1/\alpha}(y)}\right)\right)
\\&=:\Phi_t^1(x,y)+\Phi_t^2(x,y).
\end{split}
\end{equation}

We estimate $\Phi_t^1(x,y), \Phi^2(x,y)$ separately. Because $a$ is $\eta$-H\"older continuous and bounded, we have for any $\chi\leq \eta$
$$
|a(x)-a(y)|\leq C\Big(|x-y|^\chi\wedge 1\Big).
$$
Since $b$ is bounded, each $b(t, \cdot)$ is bounded by the same constant, and therefore $|\kappa_t(y)-y|\leq ct$. Then by an elementary inequality
$
|u+v|^\chi\wedge 1\leq C(|u|^\chi\wedge 1+|v|^\chi\wedge 1)$ we get finally
\begin{equation}\label{a_hol_in_c}
|a(x)-a(y)|\leq  c(|\kappa_t(y)-x|^\chi\wedge 1)+c t^\chi.
\end{equation}
   Then by  \eqref{G1},  \eqref{G2}, and \eqref{g_a_frac}
\begin{align*}
|\Phi_t^1(x,y)|&\leq Ct^{-1+\chi/\alpha} \Big| \frac{\kappa_t(y)-x}{t^{1/\alpha}}\Big|^\chi \frac{1}{t^{d/\alpha}}G^{(\alpha)}\left({\kappa_t(y)-x\over t^{1/\alpha}}\right) +Ct^{-1+\chi}\frac{1}{t^{d/\alpha}} G^{(\alpha)}\left({\kappa_t(y)-x\over t^{1/\alpha}}\right)\\
&\leq  C  t^{-1+\chi/\alpha} Q_t^{(\chi)}(x,y)+ C t^{-1+\chi} Q_t^{(0)}(x,y).
\end{align*}

Next, recall that $b(t,x)$ is Lipschitz continuous with the constant $L(t)=Ct^{\gamma/\alpha-1/\alpha}=Ct^{-1+\delta}$, and \eqref{appr} holds. Then
\be\label{appr_b}
|b(t, \kappa_t(y))-b(x)|\leq C \Big(t^{-1+\delta}|\kappa_t(y)-x|+t^{\gamma/\alpha}\Big)=C t^{-1+1/\alpha+\delta}\left({|\kappa_t(y)-x|\over t^{1/\alpha}}+1\right).
\ee
Observe that by \eqref{G1},
(\ref{G3}), and (\ref{g_a_der}) we have
\be\label{1stder}
\left({|\kappa_t(y)-x|\over t^{1/\alpha}}+1\right)
 \left|(\nabla g^{(\alpha)})\left({\kappa_t(y)-x\over t^{1/\alpha}a^{1/\alpha}(y)}\right)\right|\leq C G^{(\alpha)}\left({\kappa_t(y)-x\over t^{1/\alpha}}\right).
\ee
This gives
$$
|\Phi_t^2(x,y)|\leq  C t^{-1+\delta} Q_t^{(0)}(x,y).
$$
Combining the above estimates for $\Phi_t^1(x,y), \Phi^2(x,y)$ ,  we complete the proof.

\end{proof}

 To estimate the  convolution powers $\Phi^{\star k}_t(x,y)$, $k\geq 1$ inductively, we   modify slightly the above upper bound for $\Phi$.  For $\lambda\in [0,\alpha)$,  define
\begin{equation}\label{Ht21}
H_t^{(\lambda)} (x,y):=
\left(\left( \Big|\frac{\kappa_t(y)-x}{t^{1/\alpha}}\Big|^\lambda \vee 1 \Big) \wedge t^{-\lambda/\alpha}\right) \right)\frac{1}{t^{d/\alpha}} G^{(\alpha)} \left({\kappa_t(y)-x\over t^{1/\alpha}}\right).
\end{equation}
 Clearly,
 $$
 Q^{(\lambda)}_t(x,y)\leq H^{(\lambda)}_t(x,y),
 $$
 and therefore a (weaker) analogue of \eqref{Qt21} with  $Q^{(\chi)}_t(x,y)$, $Q^{(0)}_t(x,y)$ replaced by $H^{(\chi)}_t(x,y)$, $H^{(0)}_t(x,y)$ holds true. The advantage of this weaker estimate is that the kernels $H^{(\chi)}_t(x,y)$, $H^{(0)}_t(x,y)$ have the sub-convolution property, and therefore we can use Lemma \ref{lH10}.

 \begin{lem}  For every  $\lambda\in [0,\alpha),$ the kernel   $H_{t}^{(\lambda)}(x,y)$ has  the  sub- and super-convolution properties.
 \end{lem}

\begin{proof} We prove the sub-convolution property, only: the proof of the super-convolution one is completely analogous and is omitted. Denote
$$
K_t^{(\lambda)} (x,y)=
\left(\left( \Big|\frac{y-x}{t^{1/\alpha}}\Big|^\lambda \vee 1 \Big) \wedge t^{-\lambda/\alpha}\right) \right)\frac{1}{t^{d/\alpha}} G^{(\alpha)} \left({y-x\over t^{1/\alpha}}\right),
$$
then $K_{t}^{(\lambda)}(x,y)$ has  the  sub-convolution property, see Proposition 3.3 in \cite{KK15}, Case \textbf{A}. We have
$$
H_t^{(\lambda)} (x,y)=K_t^{(\lambda)} (x,\theta_t^0(y)).
$$
Note that  $K_t^{(\lambda)} (x,y)$  is a piece-wise power type function of $|x-y|/t^{1/\alpha}$. Then one can easily  derive from \eqref{aflow} that, for a given $T>0$, there exist positive constants $C_1, C_2$ such that for any $T\in [0, T], s\in [0, t]$
\be\label{KH}
C_1H_t^{(\lambda)} (x,y)\leq K_t^{(\lambda)} (\upsilon_{t-s}^{t-s}(x),\theta_s^0(y))\leq  C_2H_t^{(\lambda)} (x,y).
\ee
Now, in order to obtain the required property of $H_t^{(\lambda)} (x,y)$, we apply first the right hand side inequality in \eqref{KH} with $t'=t-s$ and $s'=0$, then the sub-convolution property of $K_t^{(\lambda)} (x,y)$, and then the right hand side inequality in \eqref{KH}:
$$\ba
(H_{t-s}*H_s)(x,y)&\leq C\int_{\rd}K_{t-s}^{(\lambda)} (\upsilon_{t-s}^{t-s}(x), z)K_s^{(\lambda)} (z,\theta_s^0(y))\, dz
\\&\leq C K_t^{(\lambda)}(\upsilon_{t-s}^{t-s}(x),  \theta_s^0(y))\leq C H_t^{(\lambda)} (x,y).
\ea
$$
\end{proof}

Now  we have all the conditions of Lemma \ref{lH10} satisfied with $H^1=H^{(\chi)},H^2={(0)}$, $\delta_1=\chi/\alpha,\delta_2=\delta\wedge \chi.$ In addition, we have \eqref{F11}, hence by Lemma \ref{lH10} the convolution powers $\Phi^{\star k}, k\geq 1$ are well defined and the series in the right hand side of  \eqref{sol_1} converge absolutely. In addition,  for the residue term $r_t(x,y)=p_t(x,y)-p_t^0(x,y)$ we have the upper bound \eqref{boundr} with
$$
H^1_t(x,y)= H^{(\chi)}_t(x,y), \quad H^2_t(x,y)=H^{(0)}_t(x,y),\quad \delta_1={\chi\over \alpha}, \quad \delta_2=\chi\wedge \delta.
$$
We remark that this upper bound actually gives the following estimate for $r_t(x,y)$,  similar to \eqref{r_bound}:
\begin{equation}\label{r_bound1}
|r_t(x,y)|\leq C\Big(t^\zeta+|y-\upsilon_t(x)|^\chi\wedge 1\Big)p_t^0(x,y),\quad \zeta=\min\left\{\delta, \chi, {\chi\over \alpha}\right\}.
\end{equation}
Indeed, by \eqref{g_a} we have directly
\be\label{417}
H^{(0)}_t(x,y)\leq Cg_t^0(x,y).
\ee
On the other hand,
$$\ba
H_t^{(\chi)} (x,y)&=
\left(\left( \Big|\frac{\kappa_t(y)-x}{t^{1/\alpha}}\Big|^\chi \vee 1 \Big) \wedge t^{-\chi/\alpha}\right) \right)H^{(0)}_t(x,y)
\\&\leq
\left(\left( \Big|\frac{\kappa_t(y)-x}{t^{1/\alpha}}\Big|^\chi + 1 \Big) \wedge t^{-\chi/\alpha}\right) \right)H^{(0)}_t(x,y),
\ea
$$
hence
$$
\ba
t^{\chi/\alpha}H_t^{(\chi)} (x,y)&\leq \Big((|{\kappa_t(y)-x}|^\chi+t^{\chi/\alpha} ) \wedge 1\Big)H^{(0)}_t(x,y)
\\&\leq \Big(t^{\chi/\alpha} + |{\kappa_t(y)-x}|^\chi\wedge 1\Big)H^{(0)}_t(x,y).
\ea
$$
Combined with \eqref{417} and \eqref{boundr}, this provides  \eqref{r_bound1}.

The estimates, which we have just obtained, actually yield that the series in the right hand side of \eqref{sol_1} converge uniformly for $t\in [0, T], x,y\in \rd$ for every $T$. Because $p_t^0(x,y)$ and $\Phi_t(x,y)$ are continuous in $(t,x,y)$ which possess explicit upper bounds, it is a routine calculation to show that each term $p^0\star \Phi^{\star k}, k\geq 1$ is continuous in $(t,x,y)$; e.g. \cite{KK15}, Section 3.3. This proves the assertion of Theorem \ref{t1} on  the continuity of $p_t(x,y)$. The same argument actually shows that the identity \eqref{P_int} defines a family $P_t, t>0$ of bounded linear operators on $C_\infty$, which is strongly continuous; that is,
$$
\|P_tf-P_s\|_\infty\to 0, \quad t\to s, \quad s>0, \quad f\in C_\infty.
$$

\subsection{``Fine tuning'' of the decomposition of $p_t(x,y)$ }\label{s42} We  have obtained $p_t(x,y)$ in the form \eqref{sol} with $p_t^0(x,y)$ defined by \eqref{p01}. Now we re-arrange  this decomposition to the form \eqref{decomp} and prove the estimate \eqref{r_bound} for the respective residue term.

First, we note that \eqref{appr_sol} with $s=t$ and \eqref{aflow} with $s=0$ yield
\be\label{aflow1}
e^{-Ct^\delta}|x-\kappa_t(y)|-Ct^{1/\alpha+\delta}\leq |\upsilon_{t}(x)-y|\leq e^{Ct^\delta}|x-\kappa_t(y)|+Ct^{1/\alpha+\delta}.
\ee

Next, we obtain the following property of the function $g^{(\alpha)}$.

\begin{lem}\label{l43} For any $V>0$ there exists a constant $C_V$ such that for any $v\in [0, V]$ and any $x,y\in \rd$ such that
\be\label{420}
e^{-v}|x|-v\leq |y|\leq e^{v}|x|+v,
\ee
the following inequality holds:
\be\label{421}
e^{-C_Vv}g^{(\alpha)}(x)\leq g^{(\alpha)}(y)\leq e^{C_Vv}g^{(\alpha)}(x).
\ee
\end{lem}
\begin{proof} We note that $g^{(\alpha)}$ is rotationally invariant, hence without loss of generality we can restrict ourselves to the case where $x$ and $y$ have same direction. We have
$$
{g^{(\alpha)}(x)\over g^{(\alpha)}(y)}=\exp\left\{\int_0^1\Big(\big(\nabla \log g^{(\alpha)}\big)(sx+(1-s)y), x-y\Big)\, ds\right\}.
$$
By \eqref{g_a}, \eqref{g_a_der}, and \eqref{G3},
$$
\big|\nabla \log g^{(\alpha)}(x)\big|=\left|\nabla g^{(\alpha)}(x)\over g^{(\alpha)}(x)\right|\leq C{1\over 1+|x|},
$$
hence
$$
\left|\int_0^1\Big(\big(\nabla \log g^{(\alpha)}\big)(sx+(1-s)y), x-y\Big)\, ds\right|\leq C{|x-y|\over 1+\min_{s\in [0,1]}|sx+(1-s)y|}.
$$
Because $x,y$ have the same direction, it follows from \eqref{420} that
$$
|x-y|\leq C\left((e^v-1)\min_{s\in [0,1]}|sx+(1-s)y|+v\right).
$$
Together with the previous inequality, this yields
$$
\left|\int_0^1\Big(\big(\nabla \log g^{(\alpha)}\big)(sx+(1-s)y), x-y\Big)\, ds\right|\leq C_V v,
$$
which completes the proof of the required statement.
\end{proof}

As a corollary of \eqref{aflow1} and \eqref{421}, we get
\be\label{422}
e^{-Ct^\delta}\widetilde p_t(x,y)\leq p_t^0(x,y)\leq e^{Ct^\delta}\widetilde p_t(x,y), \quad t\in (0, T],
\ee
where the constant $C$ depends on $T$, only. This yields that  the residue term $r_t(x,y)$ satisfies analogue of \eqref{r_bound1} with
$p_t^0(x,y)$ in the right hand side replaced by $\widetilde p_t(x,y)$. This also yields that
$$
|p_t^0(x,y)-\widetilde p_t(x,y)|\leq Ct^\delta \widetilde p_t(x,y), \quad t\in (0, T].
$$
Together with the above analogue of \eqref{r_bound1}, this completes the proof of \eqref{r_bound}.

The same strategy can be used in order to prove the two-sided estimate \eqref{two-sided} for $p_t(x,y)$. First, we prove similar  two-sided estimate with $p_t^0(x,y)$ instead of $\widetilde p_t(x,y)$: the upper bound is straightforward, the proof of the lower bound is completelky analogous to the proof of Theorem  2.5 \cite{KK15} and thus is omitted. Then \eqref{two-sided} follows by  \eqref{422}.

The same argument can be also applied to provide the representations of $p_t(x,y)$ outlined in Remark \ref{r21}. Namely, by \eqref{appr_sol_1} and \eqref{aflow},
$$
e^{-Ct^\delta}|x-\kappa_t(y)|-Ct^{1/\alpha+\delta}\leq |\upsilon_{t}^0(x)-y|\leq e^{Ct^\delta}|x-\kappa_t(y)|+Ct^{1/\alpha+\delta}.
$$
This yields an analogue of \eqref{r_bound}, when $\upsilon_t(x)$ in the definition of $\widetilde p_t(x,y)$ is replaced by $\upsilon_t^0(x)$.
On the other hand, \eqref{aflow} and \eqref{329} yield
$$
e^{-Ct^\delta}|x-\kappa_t(y)|-Ct^{1/\alpha+\delta}-Ct^{1/\alpha+\rho_k}\leq |\upsilon_{k,t}(x)-y|\leq e^{Ct^\delta}|x-\kappa_t(y)|+Ct^{1/\alpha+\delta}+Ct^{1/\alpha+\rho_k},
$$
which provides the required error bound in the case where $\upsilon_t(x)$ in the definition of $\widetilde p_t(x,y)$ is replaced by $\upsilon_{k,t}(x)$.

Finally, we mention that $\widetilde p_t(x,y)$ is the distribution density of $\widetilde X_{t,x}$ defined by \eqref{trotter},
and $\widetilde X_{t,x}$ weakly converge to $x$ when $t\to 0+$. Combined with \eqref{r_bound},  this means  that if we extend the family of operators $P_t, t>0$ (see  \eqref{P_int}) by the natural convention that $P_0$ is the identity operator in $C_\infty$, we get a strongly continuous family $P_t, t\geq 0$.

\subsection{Approximate fundamental solution, weak uniqueness, and the semigroup properties}\label{s43}
Here we outline the construction of the approximate fundamental solution, introduced in  \cite{KK15}, which shows that $p_t(x,y)$ constructed above indeed solves the Cauchy problem \eqref{L_fund}, \eqref{L_delta} in a certain approximate sense.  The main difficulty here is caused by the following. Because $D(L)=C^2_\infty$, to apply $L_x$ to $p_t(x,y)$ one should preliminarily specify $\nabla_x p_t(x,y)$, $\nabla_{xx}^2 p_t(x,y)$. The function $p_t(x,y)$ is presented in \eqref{sol_1}, and the naive argument would be to apply $\nabla_x$, $\nabla_{xx}^2$ to each term in this representation. The functions $\nabla_x p_t^0(x,y)$, $\nabla_{xx}^2 p_t^0(x,y)$ are explicit and well defined, however in general they exhibit strongly singular behavior (e.g. (4.1) in \cite{KK15}) which makes it impossible to apply, say, $\nabla_{xx}^2$  to any integral term in \eqref{sol_1}.

Motivated by this observation, we introduce the  family of functions $\{p_{\eps_t}(x,y), \eps>0\}$,
\begin{equation}\label{pe}
p_{t,\epsilon}(x,y)=p_{t+\epsilon}^0(x,y)  + \int_0^t \int_{\rd}  p_{t-s+\eps}^0(x,z)  \Psi_s(z,y) dzds,
\end{equation}
and define
\begin{equation} \label{Pte}
P_{t,\epsilon} f(x):=\int_\rd p_{t,\epsilon}(x,y)f(y)dy, \quad t\geq 0, \, x\in \rd, \quad f\in C_\infty(\rd).
\end{equation}
The  additional time shift by positive $\eps$ removes the singularity at the point $s=t$ and resolves the difficulty outlined above. Namely, we have the following properties, which easily follows from the explicit representation for $p_t^0(x,y)$, and the parametrix estimates for the kernels $\Phi_t(x,y)^{\star k}$; see \cite{KK15}, Lemma 4.1 for the detailed proof which actually does not relies on the specific properties of the model.

\begin{lem}\label{aux-ep}
\begin{enumerate}
\item  For every   $f\in C_\infty(\rd)$, $\eps>0$ the function $P_{t,\eps}f(x)$ belongs to $C^1$ as a function of $t$, to $C^2_\infty$ as a function of $x$, and the functions $\prt_tP_{t,\eps}f(x), L_x P_{t,\eps}f(x)$ are continuous w.r.t. $(t,x)$.
 \item    For every   $f\in C_\infty(\rd)$, $T>0$,
  \be\label{conv_pte}
 \|P_{t,\epsilon}f- P_t f\|_\infty\to 0, \quad \epsilon\to 0,
  \ee
   uniformly in $t\in [0,T ]$, and for every $\eps>0$
  \be\label{conv_x}
  P_{t,\epsilon}f(x)\to 0, \quad |x|\to \infty
  \ee
 uniformly in $t\in [0,T]$.

 \end{enumerate}
\end{lem}

The following lemma shows that the family $p_{\eps, t}(x,y), \eps>0$ satisfies \eqref{L_fund}, \eqref{L_delta} in a (weak) approximate sense.
This is the reason for us to call this family an \emph{approximate fundamental solution} to \eqref{L_fund}, \eqref{L_delta}.

Denote
\begin{equation}\label{qte20}
\Delta_{t,\eps} f(x)= \big(\partial_t-L_x\big) P_{t,\epsilon} f (x), \quad f\in C_\infty(\rd).
\end{equation}

\begin{lem}\label{l5}  For any  $f\in C_\infty(\rd)$  we have
\begin{enumerate}

  \item
  \be\label{conv_loc}
\Delta_{t,\eps} f(x)\to 0, \quad \epsilon\to 0,
\ee
uniformly in  $(t,x)\in [\tau,T]\times \rd$ for any $\tau>0$, $T>\tau$;

 \item
 $$
 \lim_{t,\eps\to 0+} \|P_{t,\eps} f-f\|_\infty =0.
  $$
\end{enumerate}
\end{lem}
The second assertion easily follows from the formula \eqref{pe}, estimates on $\Psi$ from Section \ref{s41}, and the facts that $\widetilde p_t(x,y)\to \delta_x(y), t\to0+$ and  $p^0_t(x,y)-\widetilde p_t(x,y)$ posses an estimate similar to \eqref{r_bound}. For the first assertion, we refer to the proof of Lemma 5.2 in \cite{KK15}, which is not specific with respect to the model and relies only on the smoothness properties of $p_t^0(x,y),\Phi_t(x,y)$, $\Psi_t(x,y)$, the estimates for these functions, and the fact that the function $\Psi$ satisfies the integral equation $$
\Phi_t(x,y)=\Psi_t(x,y)-\int_0^t\int_{\Re^d}\Phi_{t-s}(x,z)\Psi_s(z,y)\,dzds.
$$
All these ingredients are already available in the current setting due to the parametrix construction of Section \ref{s41}.

\begin{cor}\label{c41} Let $f\in C_\infty, T>0$ be fixed, define for $t\in [0, T], x\in \rd, \eps>0$
$$
g(t,x)=P_tf(x), \quad g_\eps(t,x)=P_{t, \eps}f(x),\quad  g^T(t,x)=P_{T-t}f(x), \quad g_\eps^T(t,x)=P_{T-t, \eps}f(x).
$$
Then for any $T>0$ \begin{itemize}
       \item for every $\eps>0$, the functions $\prt_t g_\eps(t,x), L_xg_\eps(t,x)$, $\prt_t g_\eps^T(t,x), L_xg_\eps^T(t,x)$ are well defined, continuous and bounded on $[0, T]\times \rd$;
       \item for any $\tau\in (0, T)$,
       $$
       (\prt_t-L_x)g_\eps(t,x)\to 0, \quad (\prt_t+L_x)g_\eps^T(t,x)\to 0, \quad \eps\to 0
       $$
       uniformly w.r.t. $t\in [\tau, T], x\in \rd$ and $t\in [0, T-\tau]$, respectively;
       \item
       $$
       g_\eps(t,x)\to g(t,x), \quad g_\eps^T(t,x)\to g^T(t,x), \quad \eps\to 0
       $$
       uniformly w.r.t. $t\in [0, T], x\in \rd$.
     \end{itemize}
\end{cor}

This corollary well illustrates the heuristics which motivates the notion of the approximative fundamental solution. If we were able to prove that $p_t(x,y)$ is smooth enough and is a classical fundamental solution, we would be typically able to prove that the functions $g(t,x), g^T(t,x)$ defined above are \emph{harmonic} functions for the operators $\prt_t-L$ and $\prt_t+L$, respectively. Proving the required smoothness, if even being possible, is a tough problem which would require a more delicate analysis of the structure of the initial equation. On the other hand, using the basic parametrix structure only, we are able to approximate these functions by some families in such a way that the respective operators $\prt_t-L$ and  $\prt_t+L$ on these families vanish, in a sense. This makes it possible to treat the functions $g(t,x), g^T(t,x)$ as \emph{approximate harmonic}, which appears to be quite fruitful. A good illustration for this point is given by the simple proof of the weak uniqueness of the solution to \eqref{SDE} which we explain further on.

Let $X$ be any weak solution to \eqref{SDE}, without loss of generality we can assume it has  trajectories in the space  of c\'adl\'ag functions $\DD(\Re^+, \rd)$. It follows by the It\^o formula (e.g. \cite{IW81}, Chapter II) that for any $f\in C_\infty^2$ the process
\be\label{M_f}
f(X_t)-\int_0^t Lf(X_s)\, ds
\ee
is a martingale; that is, the distribution of $X$ in $\DD(\Re^+, \rd)$ is a  solution to the martingale problem $(L, C^2_\infty)$. Hence to prove the weak uniqueness it is sufficient to show that this martingale problem is \emph{well posed} in $\DD(\Re^+, \rd)$, which means that for any probability measure $\mu$ on $\rd$ there exists at most one measure $\P$ on $\DD(\Re^+, \rd)$ such that $\P(X_0\in du)=\mu(du)$ and for any $f\in C_\infty^2$ the process \eqref{M_f} is a martingale w.r.t. $\P$. By Corollary 4.4.3 in \cite{EK86}, to do this it is sufficient to prove that for any two such measures  corresponding one-dimensional projections coincide.
Below we fix \emph{arbitrary} solution $\P$ of martingale problem $(L, C_2^\infty)$ in $\DD(\Re^+, \rd)$ with $\P(X_0\in du)=\mu(du)$ and specify its one-dimensional projections. In what follows, we denote by $\E^\P$ the expectation w.r.t. $\P$.

First, we mention that $\P$ corresponds to a stochastically continuous process. Indeed, for any pair of open sets $U, V$ such that their closures are compact and disjoint, there exists a function $f\in C_\infty^2$ such that $f(x)=0, x\in U$ and $f(x)=1, x\in V$. Then for every $s\geq 0$ we have
$$
\P(X_s\in U, X_t\in V)\leq \E^\P(f(X_t)-f(X_s))=\E^\P\int_s^t Lf(X_r)\, dr\to 0, \quad t\to s+.
$$
On the other hand, since $\P$ is a measure on a Polish space  $\DD(\Re^+, \rd)$, it is tight, and in particular for any $\eps>0$, $T>0$ there exists a compact set $K_{\eps,T}\subset \rd$ such that
$$
\P(X_s\in K_{\eps, T})\geq 1-\eps, \quad s\in [0, T];
$$
see \cite{EK86}, Section 3.5. Combining these two observations we easily get the required stochastic continuity. Now we can apply assertion (a) of Lemma 4.3.4 in \cite{EK86} and  get that for any function $g(t, x), t\in [0, T], x\in \rd$ which is differentiable w.r.t. $t$, belongs to $D(L)=C^2_\infty$ w.r.t. $x$, and has continuous
$\prt_t g(t, x), L_xg(t,x)$, the following process is a martingale w.r.t. measure $\P$:
$$
g(t, X_t)-\int_0^t\Big( \prt_sg(s, X_s)+L_xg(s, X_s)\Big)\, ds.
$$

Now all the preliminaries are complete, and we present the cornerstone of the proof. Fix any $f\in C_\infty$ and $T>0$,
and consider the function $g^T(t,x)=P_{T-t}f(x)$. If we would knew that this function is harmonic and  $\prt_t g(t, x), L_xg(t,x)$ are continuous, this would imply directly
$$
\E^\P f(X_T)\Big(=\E^\P g^T(T, X_T)=\E^\P g^T(0, X_0)\Big)=\E^\P P_{T}f(X_0).
$$
We avoid proving that $g^T(t,x)$ is smooth by using its interpretation as an approximate harmonic function, explained above. The argument remains essentially the same. Namely, we take the family $g^T_\eps(t,x), \eps>0$ defined in Corollary \ref{c41} and observe that for every $\tau\in (0, T)$
$$\ba
\E^\P P_{\tau, \eps} f(X_{T-\tau})-\E^\P P_{T, \eps}f(X_0)&=\E^\P g_{\eps}^T (T-\tau, X_{T-\tau})-\E^\P g_{\eps}^T (0, X_{0})
\\&=\E^\P \int_0^{T-\tau}(\prt_t+L_x)g^T_\eps(t,X_t)\, dt\to 0, \quad \eps\to 0.
\ea
$$
This yields
$$
\E^\P P_{\tau} f(X_{T-\tau})=\E^\P P_{T}f(X_0), \quad \tau\in (0, T).
$$
Since $X$ is stochastically continuous and $P_t, t\geq 0$ is a strongly continuous family of operators, we can take $\tau\to 0$ and get
\be\label{onedim}
\E^\P f(X_{T})=\E^\P P_{T}f(X_0).
\ee
Since $T>0, f\in C_\infty$ in the above identity are arbitrary and $X_0$ has a prescribed law $\mu$ w.r.t. $\P$, this uniquely defines
 the one-dimensional distributions of $\P$ and hence completes the proof of the fact that the martingale problem $(L, C^2_\infty)$ in $\DD(\Re^+, \rd)$ is well posed.

The only assertion left unproved in  our main Theorem \ref{t1} and Theorem \ref{t2} is that the weak solution to \eqref{SDE} exists, and it is a Markov process with the strongly continuous semigroup in $C_\infty$ defined by \eqref{P_int}.

The existence of the weak solution follows by the standard approximation argument: approximate $b, \sigma$ by smooth $b^n, \sigma^n, n\geq 1$, consider the corresponding strong solutions $X^n, n\geq 1$, and prove that (a) the sequence $(X^n, Z^{(\alpha)})$  is weakly compact in $\DD(\Re^+, \rd\times \rd)$; (b) any weak limit point of this sequence is a weak solution to \eqref{SDE}. We refer to Section 5 in \cite{KK15} for a detailed exposition of this argument in a similar setting.

We have just proved that the solution to the martingale problem $(L, C^2_\infty)$ in $\DD(\Re^+, \rd\times \rd)$ both exists and unique, and corresponds to the unique weak solution
to \eqref{SDE}. Now we use the general fact that if a martingale problem in $\DD(\Re^+, \rd)$ is well posed, then (under a technical measurablity assumption, see Theorem 4.2.4 (c), \cite{EK86}, which now easily  follows by Theorem 4.2.6, \cite{EK86}) its solution $X$ is a strong Markov process for every $s<t$
\be\label{markov}
\P\Big(X_t\in A\Big|\mathcal{F}_s\Big)=\P_{X_s}(X_{t-s}\in A), \quad A\in \mathcal{B}(\rd);
\ee
here we denote by $\mathcal{F}_s, s\geq 0$ the natural filtration on $\DD(\Re^+, \rd)$ and by $\P_x, x\in \rd$ the solution to the martingale problem with the initial distribution $\mu(du)=\delta_x(du)$.  Denote $P_t(x,dy)=p_t(x,y)dy, t>0, P_0(x,dy)=\delta_0(dy),  x\in \rd$, and observe that by  \eqref{onedim} and \eqref{markov} applied to $\P$ with $\P(X_0=x)=1$, for any $f\in C_\infty$:
$$
\int_{\rd} f(y) P_t(x,dy)=\E^\P f(X_t)=\E^\P\int_{\rd}f(y)P_{t-s}(X_s,dy)=\int_{\rd}\int_{\rd} f(y) P_{t-s}(z,dy)P_s(x,dz).
$$
This proves the Chapman-Kolmogorov equality for $P_t(x, dy), t\geq 0$. Hence $P_t, t\geq 0$ is a semigroup  of operators in $C_\infty$, recall that we have already proved that it is strongly continuous. Applying \eqref{onedim} and \eqref{markov} once again we derive that $P_t(x,dy),t\geq 0$ is the transition probability density for $X$, and respectively $P_t, t \geq 0$ is the (Feller) semigroup corresponding to this Markov process, and in particular each operator $P_t, t\geq 0$ is non-negative. Because $\P_x(\DD(\Re^+, \rd\times \rd))=1, x\in \rd$, this semigroup is conservative.
\qed

\subsection{Concuding remark: other possible choices for $p_t^0(x,y)$}\label{choice}

Now when the proof is complete and its structure is clearly visible, we can discuss some other possibilities for the choice of the ``zero order approximation'' in \eqref{sol}. One such a choice is \eqref{p01} with $\kappa_t(y)$ equal to an exact solution to \eqref{backwardODE}. Under such a choice the formula \eqref{Phi_c} for $\Phi$ becomes even simpler: $b(t,\kappa_t(y))$ changes to $b(\kappa_t(y))$. All the subsequent calculation remain literally the same, with the definition of the kernels $H^{(\lambda)}_t(x,y)$ respectively changed. A subtle technical point here is that for a given $y$ exact solution to \eqref{backwardODE} is not unique and we have take to care about choosing the function $\kappa_t(y)$ in a measurable way. This difficulty however is not substantial and can be resolved by using a proper version of the Kuratovskii and Ryll-Nardzevski  measurable selection theorem, e.g. \cite{Stroock_Varad}, Theorem 12.1.10. Another non-trivial part here is to prove the sub-convolution property for the modified kernels $H^{(\lambda)}_t(x,y)$, but this can be made in a completely similar way as we did that for original kernels; here one should use  an easy consequence of \eqref{aflow} that the same two-sided inequality holds true for any pair of exact solutions  to \eqref{forwardODE} and \eqref{backwardODE}.

Similarly, for $k$ such that \eqref{211} holds one can take $\kappa_t(y)$ equal to the $k$-th Picard type approximation $\theta_{k,t}(y)$ for \eqref{backwardODE}. Under such a choice,  the measurability issues do not arise, and the same calculation as above can be made, with extra error terms with come from the difference between $\theta_{k,t}(y)$ and $\theta(y)$. Due to these terms the index $\zeta$ in the upper bound for the corresponding error term will be changed to $\min\{\zeta, \rho_k/(1-\gamma)\}$.

Hence, there is no visible difference between these two possibilities and the one we used above, if one is aimed only at the basic properties of the solution to \eqref{SDE} (weak existence and uniqueness and the estimates for the transition probability density itself). The situation is changed drastically when the \emph{sensitivity} of the transition probability density is studied: in Remark \ref{r51} below we indicate the point which makes our choice of $\kappa_t(y)$ apparently the only possible one to use in the study of the derivative $\prt_tp_t(x,y)$.

\section{Proof of Theorem~\ref{t3}}\label{s5}
The proof repeats the general lines of the proof of Theorem 2.6 \cite{KK15}. Hence here we only outline the main steps  of the proof and explain in details the key point which makes it possible both to improve the statement of  Theorem 2.6 \cite{KK15}; that is to  cover the case of $b(x)$ which is not Lipschitz continuous. We omit all the technicalities for which it is possible to refer to  analogous parts of the proof of Theorem 2.6 \cite{KK15}.

First,  it is a direct calculation to show that $p_t^0(x,y)$ defined by \eqref{p01} is continuously differentiable w.r.t. $t\in (0, \infty$ and    $$
   |\prt_tp_t^0(x,y)|\leq Ct^{-1/\alpha'}H_t^{(0)}(x,y),\quad  t\in (0, T], \quad x,y\in \Re^d
   $$
for any $T>0$; cf. Proposition 4.1 \cite{KK15}. Our goal is to extend these properties of $p_t^0(x,y)$ to similar properties of $p_t(x,y)$. We have the integral representation
\be\label{sol_2}
p_t(x,y)=p_t^0(x,y)+\int_0^t\int_{\Re^d}p_{t-s}^0(x,z)\Psi_s(z,y)\,dzds,
\ee
which is just the other form of \eqref{sol_1}; cf. \eqref{Psi}. Because of the non-integrable singularity $(t-s)^{-1/\alpha'}$ of the upper bounnd for the function $\prt_tp_{t-s}^0(x,z)$, we can not take $\prt_t$ at the right hand side of \eqref{sol_2} directly. Instead of that, we do the following standard trick (e.g. proof of Theorem 2.3 in \cite{KM02}). Rewrite \eqref{sol_2} in the following way:
 \be\label{46}\ba
p_t(x,y)&
=p_t^0(x,y)+\int_0^{t/2}\int_{\Re^d}p_{t-s}^0(x,z)\Psi_s(z,y)\,dzds+\int_{0}^{t/2}\int_{\Re^d}p_{s}^0(x,z)\Psi_{t-s}(z,y)\,dz\, ds,
\ea
\ee
then we avoid the singularity at $s=t$ related to $p_t^0(x,y)$, but instead we have to establish the differential properties of $\Psi$ with respect to $t$. This can be done in the same fashion with the basic parametrix method. Namely, we will study the differential properties of $\Phi$ and then extend them to $k$-fold convolutions $\Phi^{\star k}, k\geq 1$ and their sum $\Psi$. We explain in details the first step, which contains the most subtle point. In what follows we use the  notation from Section \ref{s41}.

\begin{lem}\label{deriv_Phi}  The function $\Phi_t(x,y)$ defined by (\ref{Phi})  has a continuous  derivative
$
\prt_t\Phi_t(x,y),
$ on $(0,\infty)\times\Re^d\times\Re^d$. In addition, for any $\chi\in (0, \eta\wedge \alpha)$ and $T>0$
   \begin{equation}\label{phider1}
   |\prt_t\Phi_t(x,y)|\leq C t^{-1/\alpha'}\Big(t^{-1+\chi/\alpha} Q_t^{(\chi)}+ (t^{-1+\chi} + t^{-1+\delta})Q_t^{(0)}(x,y)\Big),\quad  t\in(0, T], \quad x,y\in \Re^d.
   \end{equation}
\end{lem}

\begin{proof}  We take $\prt_t$ separately for $\Phi_t^1(x,y)$ and $\Phi_t^2(x,y)$; see (\ref{Phi_c}). Obviously, both these derivatives are well defined and continuous. Next, we have
$$\ba
|\prt_t\Phi_t^1(x,y)|&\leq C\Big|a(x)-a(y)\Big|\bigg\{{1\over t^{d/\alpha+2}}\Big|(L^{(\alpha)} g^{(\alpha)})\left({\kappa_t(y)-x\over t^{1/\alpha}a^{1/\alpha}(y)}\right)\Big|\\&\hspace*{1cm}+{1\over t^{d/\alpha+1+1/\alpha}}\Big|(\nabla L^{(\alpha)} g^{(\alpha)})\left({\kappa_t(y)-x\over t^{1/\alpha}a^{1/\alpha}(y)}\right)\Big||\prt_t\kappa_t(y)|\\&\hspace*{1cm}+{1\over t^{d/\alpha+2}}\Big|(\nabla L^{(\alpha)} g^{(\alpha)})\left({\kappa_t(y)-x\over t^{1/\alpha}a^{1/\alpha}(y)}\right)\Big|\left|{\kappa_t(y)-x\over t^{1/\alpha}}\right|\bigg\},
\ea
$$
 Now the estimates are similar to those which we used for $\Phi^1_t(x,y)$ itself. Recall that $\prt_t\kappa_t(y)=-b(t, \kappa_t(y))$, which is bounded. Then  by \eqref{G1}--\eqref{G3}, (\ref{g_a_frac}),  (\ref{g_a_frac_der}), and  (\ref{a_hol_in_c}) we get
\begin{align*}
|\prt_t\Phi_t^1(x,y)|&\leq C \Big|a(y)-a(x)\Big|
 \Big({1\over t^{d/\alpha+2}} +  {1\over t^{d/\alpha+1+1/\alpha}} \Big)G^{(\alpha)}\left({\kappa_t(y)-x\over t^{1/\alpha}}\right)\\
 &\leq C t^{-1/\alpha'}\Big( t^{-1+\chi/\alpha}Q_t^{(\chi)}(x,y)+t^{-1+\chi} Q_t^{(0)}(x,y)\Big).
\end{align*}
Next,
$$\ba
|\prt_t\Phi_t^2(x,y)|&\leq Ct^{-d/\alpha-1/\alpha-1}\Big|b(t,\kappa_t(y))-b(x)\Big|\Big|(\nabla g^{(\alpha)})\left({\kappa_t(y)-x\over t^{1/\alpha}a^{1/\alpha}(y)}\right)\Big|\\ &+Ct^{-d/\alpha-1/\alpha-1}\Big|b(t,\kappa_t(y))-b(x)\Big|\Big|(\nabla^2 g^{(\alpha)})\left({\kappa_t(y)-x\over t^{1/\alpha}a^{1/\alpha}(y)}\right)\Big|\left|{\kappa_t(y)-x\over t^{1/\alpha}}\right|\\
 &+Ct^{-d/\alpha-2/\alpha}\Big|b(t,\kappa_t(y))-b(x)\Big|\Big|(\nabla^2 g^{(\alpha)})\left({\kappa_t(y)-x\over t^{1/\alpha}a^{1/\alpha}(y)}\right)\Big||\prt_t\kappa_t(y)|
 \\&+ Ct^{-d/\alpha-1/\alpha}\Big|(\nabla g^{(\alpha)})\left({\kappa_t(y)-x\over t^{1/\alpha}a^{1/\alpha}(y)}\right)\Big|\Big|\prt_t\Big(b(t, \kappa_t(y))\Big)\Big|=:\sum_{j=1}^4 \Upsilon_t^{j}(x,y).
\ea
$$
By \eqref{appr_b} and \eqref{1stder}, we have
$$
\Upsilon_t^{1}(x,y)\leq Ct^{-2+\delta} Q_t^{(0)}(x,y).
$$
Next, similarly to \eqref{1stder} we have
$$
\left({|\kappa_t(y)-x|\over t^{1/\alpha}}+1\right)^2
 \left|(\nabla^2 g^{(\alpha)})\left({\kappa_t(y)-x\over t^{1/\alpha}a^{1/\alpha}(y)}\right)\right|\leq C G^{(\alpha)}\left({\kappa_t(y)-x\over t^{1/\alpha}}\right),
 $$
see Proposition \ref{A1}. Then using  \eqref{appr_b} and  the fact that $\prt_t\kappa_t(y)$ is bounded,  we get
$$
\Upsilon_t^{2}(x,y)\leq Ct^{-2+\delta} Q_t^{(0)}(x,y), \quad \Upsilon_t^{3}(x,y)\leq Ct^{-1-1/\alpha+\delta} Q_t^{(0)}(x,y).
$$
Finally, we have similarly to \eqref{appr_der}
$$
|\prt_t b(t,x)|\leq C t^{-d/\alpha}\int_{\rd}|z-x|^{\gamma}\left({|z-x|^2\over t^{2/\alpha+1}}+{1\over t}\right)e^{-|z-x|^2/2t^{2/\alpha}}\, dz\leq C t^{\gamma/\alpha-1}=Ct^{-2+\delta+1/\alpha}.
$$
Because
$$
\prt_t\Big(b(t, \kappa_t(y))\Big)=(\prt_tb)(t, \kappa_t(y))+\Big(\nabla b(t, \kappa_t(y)), \prt_t\kappa_t(y)\Big)
$$
and $\prt_t\kappa_t(y)$ is bounded, we get by \eqref{appr_der}
$$
\Big|\prt_t\Big(b(t, \kappa_t(y))\Big)\Big|\leq C\Big(t^{-2+\delta+1/\alpha}+t^{-1+\delta}\Big).
$$
Then by \eqref{G2}, \eqref{g_a_der},
$$
\Upsilon_t^{4}(x,y)\leq C\Big(t^{-2+\delta}+t^{-1+\delta-1/\alpha}\Big) Q_t^{(0)}(x,y).
$$
Summarizing the above estimates for $\Upsilon^j_t(x,y), j=1, \dots, 4$ we get finally
$$
|\prt_t\Phi_t^2(x,y)|\leq Ct^{-1/\alpha'}t^{1-\delta}Q_t^{(0)}(x,y).
$$
\end{proof}
\begin{rem}\label{r51} Now we can emphasize the only, but important point which makes our current choice of the term $\kappa_t(y)$ in the ``zero order approximation'' \eqref{p01} better than all the others discussed in Section \ref{choice} above. In the calculation which corresponds to $\Upsilon^j_t(x,y)$, the following term arise:
$$
\Big(\nabla b(t, \kappa_t(y)), \prt_t\kappa_t(y)\Big).
$$
If we take $\kappa_t(y)$ equal to an exact solution to \eqref{backwardODE}, a similar term has the form
$$
\Big(\nabla b(\kappa_t(y)), \prt_t\kappa_t(y)\Big),
$$
which is badly defined because $b$ is  assumed to be H\"older continuous, only. The same problem arises if one takes $\kappa_t(y)$ equal $\theta_{k,t}(y)$, the $k$-th Picard type iteration for \eqref{backwardODE}. This illustrates the point at which  the ``mollifying of the drift coefficient'' procedure, adopted in our construction, is particularly useful: it improves the smoothness property of the drift coefficient without loss of the accuracy of the method.
\end{rem}

Now the rest of the proof is completely similar to the proof of Theorem 2.6 \cite{KK15}, and we just outline the main steps. Clearly, we can weaken \eqref{phider1} by changing $Q^{(\chi)}, Q^{(0)}$ therein to $H^{(\chi)}, H^{(0)}$. We  re-arrange the formula for the $k$-fold convolution power of $\Phi$ in the form similar to \eqref{46}:
$$
\Phi^{\star (k)}_t(x,y)=\int_0^{t/2}\int_{\Re^d}\Phi_{t-s}^{\star (k-1)}(x,z)\Phi_s(z,y)\,dz ds+\int_0^{t/2}\int_{\Re^d}\Phi_{s}^{\star (k-1)}(x,z)\Phi_{t-s}(z,y)\,dz ds.
$$
Then by induction we show that  each $\Phi^{\star k}_t(x,y)$ is continuously differentiable in  $t$ and for each $T>0$ there exist $C_1, C_2$ such that for $t\in (0, T]$
$$
 |\prt_t\Phi^{\star k}_t(x,y)|\leq {C_1 C_2^k \over \Gamma(k\zeta)} t^{-1/\alpha'} t^{(k-1)\zeta}\Big( t^{-1+\chi/\alpha}H_t^{(\chi)}(x,y)+ (t^{-1+\chi}+t^{-1+\delta})H_t^{(0)}(x,y)\Big),\quad k\geq 1;
 $$
cf. Lemma 6.2 \cite{KK15}.  Then the sum $\Psi$ of these convolution powers possesses the same properties: it has a continuous derivative in $t$ and $$
 |\prt_t\Psi_t(x,y)|\leq C t^{-1/\alpha'} \Big( t^{-1+\chi/\alpha}H_t^{(\chi)}(x,y)+(t^{-1+\chi}+t^{-1+\delta})H_t^{(0)}(x,y)\Big),\quad t\in (0, T].
$$
This allows us to take a derivative $\prt_t$ at the right hand side of \eqref{46}, and to prove that there exists continuous derivative $\prt_tp_t(x,y)$ with
$$
|\prt_tp_t(x,y)-\prt_tp_t^0(x,y)|\leq Ct^{-1/\alpha'}\Big( t^{\chi/\alpha}(H^{(\chi)}\star H^0)_t(x,y)+
(t^{1+\chi}+t^{1+\delta})(H^{(0)}\star H^0)_t(x,y)\Big).
$$
Because $H^{(\chi)}\geq H^{(0)}$ and $H^{(\chi)}, H^{(0)}$ have the sub-convolution property, we get then
$$
|\prt_tp_t(x,y)-\prt_tp_t^0(x,y)|\leq Ct^{-1/\alpha'}\leq Ct^{-1/\alpha'}\Big( t^{\chi/\alpha}H^{(\chi)}_t(x,y)+
(t^{1+\chi}+t^{1+\delta})H^0_t(x,y)\Big),
$$
which after a simple re-arrangement yields an estimate on $|\prt_tp_t(x,y)-\prt_tp_t^0(x,y)|$ similar to the estimate for $\widetilde R_t(x,y)$ in the assertion of the theorem, with $\widetilde p_t(x,y)$ in the right hand side replaced by $p_t^0(x,y)$. Repeating the ``fine tuning'' argument from the Section \ref{s42}, we complete the proof. \qed

\textbf{Acknowledgement.}  The author thanks   R. Schilling, A.Kohatsu-Higa, and K.Bogdan   for inspiring discussions  and helpful remarks,  and A.Pilipenko for bringing the paper \cite{TTW74} to his attention. The DFG Grant Schi~419/8-1 and the joint DFFD-RFFI project No. 09-01-14
and gratefully acknowledged.

   \end{document}